\documentclass[11pt]{article}

\usepackage{epsfig}
\usepackage{graphicx}
\usepackage{amsbsy}
\usepackage{amsmath}
\usepackage{amsfonts}
\usepackage{amssymb}
\usepackage{textcomp}
\usepackage{hyperref}
\usepackage{aliascnt}

\newcommand{\mcm}[3]{\newcommand{#1}[#2]{{\ensuremath{#3}}}} 

\mcm{\tuple}{1}{\langle #1 \rangle}
\mcm{\name}{1}{\ulcorner #1 \urcorner}
\mcm{\Nbb}{0}{\mathbb{N}}
\mcm{\Zbb}{0}{\mathbb{Z}}
\mcm{\Rbb}{0}{\mathbb{R}}
\mcm{\Cbb}{0}{\mathbb{C}}
\mcm{\Qbb}{0}{\mathbb{Q}}
\mcm{\Bcal}{0}{\cal B}
\mcm{\Ccal}{0}{\cal C}
\mcm{\Dcal}{0}{\cal D}
\mcm{\Ecal}{0}{\cal E}
\mcm{\Fcal}{0}{\cal F}
\mcm{\Gcal}{0}{\cal G}
\mcm{\Hcal}{0}{\cal H}
\mcm{\Ical}{0}{\cal I}
\mcm{\Jcal}{0}{\cal J}
\mcm{\Kcal}{0}{\cal K}
\mcm{\Lcal}{0}{\cal L}
\mcm{\Mcal}{0}{\cal M}
\mcm{\Ncal}{0}{\cal N}
\mcm{\Ocal}{0}{{\cal O}}
\mcm{\Pcal}{0}{{\cal P}}
\mcm{\Qcal}{0}{{\cal Q}}
\mcm{\Rcal}{0}{{\cal R}}
\mcm{\Scal}{0}{{\cal S}}
\mcm{\Tcal}{0}{{\cal T}}
\mcm{\Ucal}{0}{{\cal U}}
\mcm{\Vcal}{0}{{\cal V}}
\mcm{\Wcal}{0}{{\cal W}}
\mcm{\Xcal}{0}{{\cal X}}
\mcm{\Ycal}{0}{{\cal Y}}
\mcm{\Mfrak}{0}{\mathfrak M}

\mcm{\restric}{0}{\upharpoonright}
\mcm{\upset}{0}{\uparrow}
\mcm{\onto}{0}{\twoheadrightarrow}
\mcm{\smallNbb}{0}{{\small \mathbb{N}}}
\DeclareMathOperator{\preop}{op}
\mcm{\op}{0}{^{\preop}}

{\begin{array}{c}
\setlength{\unitlength}{1em}}%
{\end{array}}

\usepackage{amsthm}

\newcommand{\theoremize}[2]{\newaliascnt{#1}{thm} \newtheorem{#1}[#1]{#2} \aliascntresetthe{#1}}

\theoremstyle{plain}
\newtheorem{thm}{Theorem}[section]
\theoremize{lem}{Lemma}
\theoremize{sublem}{Sublemma}
\theoremize{claim}{Claim}
\theoremize{rem}{Remark}
\theoremize{prop}{Proposition}
\theoremize{cor}{Corollary}
\theoremize{que}{Question}
\theoremize{oque}{Open Question}
\theoremize{con}{Conjecture}

\theoremstyle{definition}
\theoremize{dfn}{Definition}
\theoremize{eg}{Example}
\theoremize{exercise}{Exercise}
\theoremstyle{plain}

\usepackage{epstopdf}
\usepackage{verbatim}
\usepackage{enumerate}
\usepackage[all]{xy}
\usepackage[percent]{overpic}
\usepackage{authblk}

\usepackage{subfig}

\title{A Liouville hyperbolic souvlaki}

\author[1]{Johannes Carmesin}
\author[2]{Bruno Federici\thanks{Supported by an EPSRC grant EP/L505110/1.}}
\author[3]{Agelos Georgakopoulos\thanks{Supported by EPSRC grant EP/L002787/1, and by the European Research Council (ERC) under the European Union’s Horizon 2020 research and innovation programme (grant agreement No 639046).}}
\affil[1]{University of Cambridge, Wilberforce Road CB3 0WB}
\affil[2,3]{{Mathematics Institute}\\
	{University of Warwick}\\
	{CV4 7AL, UK}\\}

\newcommand{\sm}{\setminus}
\newcommand{\comm}[1]{}

\newcommand{\labtequ}[2]{
 \begin{equation} \label{#1} 	\begin{minipage}[c]{0.9\textwidth}  #2 \end{minipage} \ignorespacesafterend \end{equation} }

\newcommand{\N}{\ensuremath{\mathbb N}}

\newcommand{\defi}[1]{{\emph{#1}}}

\newcommand{\rw}{random walk}

\renewcommand{\iff}{if and only if}
\newcommand{\fe}{for every}
\newcommand{\Fe}{For every}

\newcommand{\st}{such that}

\newcommand{\ti}{there is}

\newtheorem{problem}{{Problem}}

\newcommand{\shf}{sharp harmonic function}

\newcommand{\labtequstar}[1]{ \begin{equation*}  	\begin{minipage}[c]{0.9\textwidth}  #1 \end{minipage} \ignorespacesafterend \end{equation*} }

\newcommand{\ca}{\ensuremath{\mathcal A}}

\newcommand{\PrI}[1]{\ensuremath{\mathbb{P}\left[#1\right]}}

\newcommand{\PrII}[2]{\ensuremath{\mathbb P_{#1}\left[#2\right]}}
\newcommand{\del}{\ensuremath{\delta}}

\newcommand{\eps}{\ensuremath{\epsilon}}

\newcommand{\Lr}[1]{Lemma~\ref{#1}}

\newcommand{\Tr}[1]{Theorem~\ref{#1}}

\newcommand{\Sr}[1]{Section~\ref{#1}}

\newcommand{\Cr}[1]{Corollary~\ref{#1}}
\newcommand{\Cnr}[1]{Con\-jecture~\ref{#1}}

\newcommand{\Ex}{\mathbb E}
\newcommand{\coc}{cocircular}
\newcommand{\fig}[1]{Figure~\ref{#1}}

\begin{document}
\maketitle
	
\begin{abstract}
We construct a transient bounded-degree graph no transient subgraph of which embeds in any surface of finite genus.

Moreover, we construct a transient, Liouville, bounded-degree, Gromov--hyperbolic graph with trivial hyperbolic boundary that has no transient subtree. This answers a question of Benjamini. This graph also yields a (further) counterexample to a conjecture of Benjamini and  Schramm.

\end{abstract}

\section{Introduction}

A well-known result of Benjamini \& Schramm \cite{BeSchrChee} states that every non-amenable graph contains a non-ambenable tree. This naturally motivates seeking for other properties that imply a subtree with the same property. However, there is a simple example of a transient graph that does not contain a transient tree \cite{BeSchrChee} (such a graph had previously also been obtained by McGuinness \cite{McGuinness}). We improve this by constructing ---in \Sr{Kr}---  a transient bounded-degree graph no transient subgraph of which embeds in any surface of finite genus (even worse, every transient subgraph has the complete graph $K^r$ as a minor for every $r$). This answers a question of I.~Benjamini (private communication).

Given these examples, it is natural to ask for conditions on a transient graph that would imply a transient subtree. In this spirit, Benjamini \cite[Open~Problem~1.62]{BenCoa} asks whether hyperbolicity is such a condition. We answer this in the negative by constructing ---in \Sr{subtree}--- a transient hyperbolic (bounded-degree) graph that has no transient subtree. While preparing this manuscript, T.~Hutchcroft and A.~Nachmias (private communication) provided a simpler example with these properties, which we sketch in \Sr{subtree2}. 

A related result of Thomassen states that if a graph satisfies a certain isoperimetric inequality, then it must have a transient subtree \cite{ThoIso}.

\medskip
The starting point for this paper was the following problem of Benjamini and  Schramm


\begin{con}[{\cite[1.11.~Conjecture]{BeSchrHar}}] \label{bsconj}
Let $M$  be a connected, transient, Gromov-hyperbolic, Riemannian manifold with bounded local geometry, with the property that the union of
all bi-infinite geodesics meets every ball of sufficiently large radius. Then $M$  admits non constant bounded harmonic functions. Similarly, a Gromov-hyperbolic
bounded valence, transient graph, with $C$-dense bi-infinite geodesics has non constant bounded harmonic functions.
\end{con}

We remark that in order to disprove ---the second sentence of--- this, it suffices to find a transient, Gromov-hyperbolic
bounded valence (aka. degree)  graph with the Liouville property
; for given such a graph $G$, one can attach a disjoint 1-way infinite path to each vertex of $G$, to obtain a graph having 1-dense bi-infinite geodesics while preserving all other properties. 
As pointed out by I.~Benjamini (private communication), it is not hard to prove that any `lattice' in a horoball in 4-dimensional hyperbolic space has these properties. We prove that our example also  has these properties, thus providing a further counterexample to \Cnr{bsconj}. A perhaps surpising aspect of our example is that all of its geodesics eventually coincide despite its transience; see \Sr{secConstr}.

Although we do not formally provide a counterexample to the first sentence of  \Cnr{bsconj}, we believe it is easy to obtain one by blowing up the edges of our graph into tubes. 

In \Sr{sketch} we provide a sketch of this construction, from which the expert reader might be able to deduce the details.

\comm{
\subsection{old}

In this paper we construct a bounded-degree (Gromov-)hyperbolic graph with trivial hyperbolic boundary, on which \rw\ is transient. The surprise here is the combination of transience and having a trivial hyperbolic boundary.

In particular, although all infinite geodesics of such a graph must have bounded pairwise distance, and in fact all of them eventually coincide in our graph, this does not imply that there is an infinite sequence of bounded vertex sets separating an arbitrarily large vertex set from infinity.

In addition, we show that our construction admits no non-constant bounded harmonic functions
}

\section{The hyperbolic Souvlaki} \label{secConstr}

In this section we construct a  bounded-degree graph $\Psi$ with the following properties
\begin{enumerate}
	\item it is hyperbolic, and its hyperbolic boundary consists of a single point;
	\item for every vertex $x$ of $\Psi$, there is a unique infinite geodesic starting at $x$, and any two 1-way infinite geodesics of $\Psi$ eventually coincide;
	\item it is transient;
	\item every subtree of $\Psi$ is recurrent;
	\item it has the Liouville property.
\end{enumerate}

This graph thus yields a counterexample to \cite[Open~Problem~1.62]{BenCoa} and \Cnr{bsconj} as mentioned in the Introduction.

\subsection{Sketch of construction} \label{sketch}

Let us sketch the construction of this graph $\Psi$, and outline the reasons why it has the above properties. It consists of an 1-way infinite path $S=s_0 s_1 \ldots\,$, on which we glue a sequence $M_i$ of finite increasing subgraphs of an infinite `3-dimensional' hyperbolic graph $H_3$. For example, $H_3$ could be the 1-skeleton of a regular tiling of 3-dimensional hyperbolic space, and the $M_i$ could be taken to be  copies of balls of increasing radii around some origin in $H_3$, although it was more convenient for our proofs to construct  different $H_3$ and $M_i$. 

In order to glue $M_i$ on $S$, we identify the subpath $s_{2^i} \ldots s_{2^{i+2}-1}$ with a geodesic of the same length in $M_i$. Thus $M_i$ intersects $M_{i-1}$ and $M_{i+1}$ but no other $M_j$, and this intersection is a subpath of $S$; see Figure \ref{Fig:meatballs}. (Our graph can be quasi-isometrically embedded in $\mathbb H^5$, but probably not in $\mathbb H^4$.) We call this graph a \defi{hyperbolic souvlaki}, with \defi{skewer} $S$ and \defi{meatballs} $M_i$. We detail its construction in \Sr{secConstr}.

To prove that this graph is transient, we construct a flow of finite energy from $s_0$ to infinity (\Sr{secTrans}). This flow carries a current of strength $2^{-i}$ inside $M_i$ out of each vertex in $s_{2^i} \ldots s_{2^{i+1}-1}$, and distributes it evenly to the vertices in $s_{2^{i+1}} \ldots s_{2^{i+2}}$  for every $i$. These currents can be thought of as flowing on spheres of varying radii inside $M_i$, avoiding each other, and it was important to have at least three dimensions for this to be possible while keeping the energy dissipated under control.

To prove that our graph has the Liouville property, we observe that \rw\ has to visit $S$ infinitely often, and has enough time to `mix' inside the $M_i$ between subsequent visits to $S$ (\Sr{secLiou}).

\subsection{Formal construction}

We now explain our precise construction, which is similar but not identical to the above sketch.
We start by constructing a hyperbolic graph $H_3$ which we will use as a model for the `meatballs' $M_i$; more precisely, the $M_i$ will be chosen to be increasing subgraphs of $H_3$.

Let $T_3$ denote the infinite tree with one vertex $r$, which we call the \defi{root}, of degree 3 and all other vertices of degree 4. For $n=1,2,\ldots,$ we put a cycle ---of length $3^n$--- on the vertices of $T_3$ that are at distance $n$ from $r$ in such a way that the resulting graph is planar\footnote{Formally, we pick a cyclic ordering on the neighbours of $r$ and a linear ordering on the outer neighbours of every other vertex of $T_3$. 
Given a cyclic ordering on the vertices at level $n$ of $T_3$, we get a cyclic ordering at level $n+1$ by replacing each vertex by the linear ordering on its outer neighbours. Now we add edges between any two vertices that are adjacent in any of these cyclic orderings.}; see Figure \ref{Fig:H2}. 
We denote this graph by $H_2$. It is not hard to see that $H_2$ is hyperbolic.

\begin{figure}[h!]
\begin{center}
\includegraphics[scale=0.8]{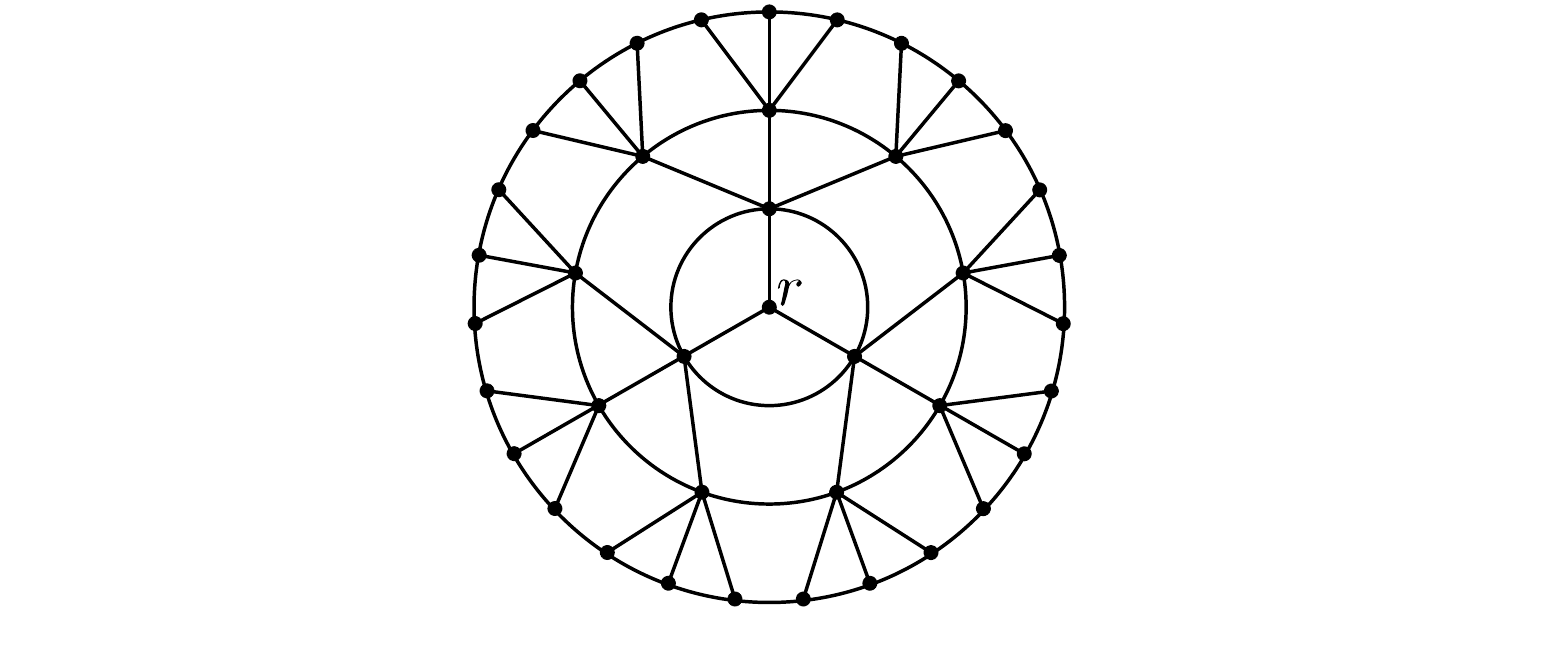}
\caption{The ball of radius 3 around the root of $H_2$.}\label{Fig:H2}
\end{center}
\end{figure}

Recall that a \emph{ray} is a 1-way infinite path.
We will now turn $H_2$ into a `3-dimensional' hyperbolic graph $H_3$, in such a way that each ray inside $T_3$ (or $H_2$) starting at $r$ gives rise to a subgraph of $H_3$ isomorphic to the graph $W$ of Figure \ref{Fig:ww}, which is a subgraph of the Cayley graph of the Baumslag-Solitar group $BS(1,2)$. Formally, we construct $W$ from infinitely many vertex disjoint double rays\footnote{A \emph{double ray} is a 2-way infinite path.} $D_0, D_1, D_2,..$, where $D_i=...r_i^{-2}r_i^{-1}r_i^0r_i^1r_i^2...$.
Then we add all edges of the form $r_i^{k}r_{i+1}^{2k}$. 

\begin{figure}[h!] 
\begin{center}
\includegraphics[scale=0.2]{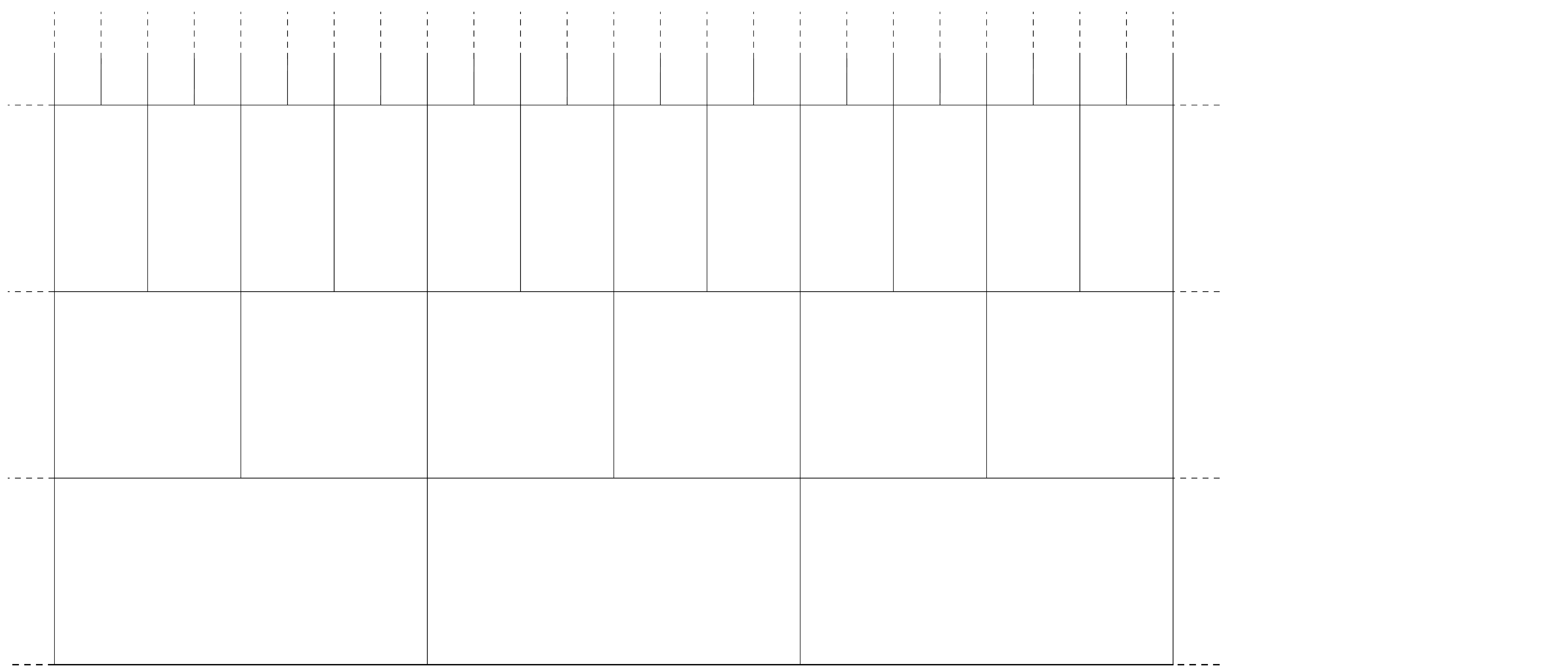}
\caption{The graph $W$: a subgraph of the standard  Cayley graph of the Baumslag-Solitar group $BS(1,2)$. It is a plane hyperbolic graph.}\label{Fig:ww}
\end{center}
\end{figure}

To define $H_3$, we let the  \defi{height} $h(t)$ of a vertex $t\in V(H_2)$ be its distance $d(r,t)$ from the root $r$. For a vertex $w$ of $W$, we say that its height $h(w)$ is $n$ if $w$ lies in $D_n$, the $n$th horizontal double ray in \fig{Fig:ww}. 

We define the vertex set of $H_3$ to consist of all ordered pairs $(t,w)$ where $t$ is a vertex of $H_2$ and $w$ is a vertex of $W$ and $h(w)= h(t)$. The edge set of $H_3$ consists of all pairs of pairs $(t,w) (t',w')$ \st\ either 

\begin{itemize}
\item $tt'\in E(H_2)$ and $ww'\in E(W)$, or 
\item $tt'\in E(H_2)$ and $w=w'$, or 
\item $t=t'$ and $ww'\in E(W)$.
\end{itemize}

\begin{figure}[h!]
\begin{center}
\includegraphics[scale=0.2]{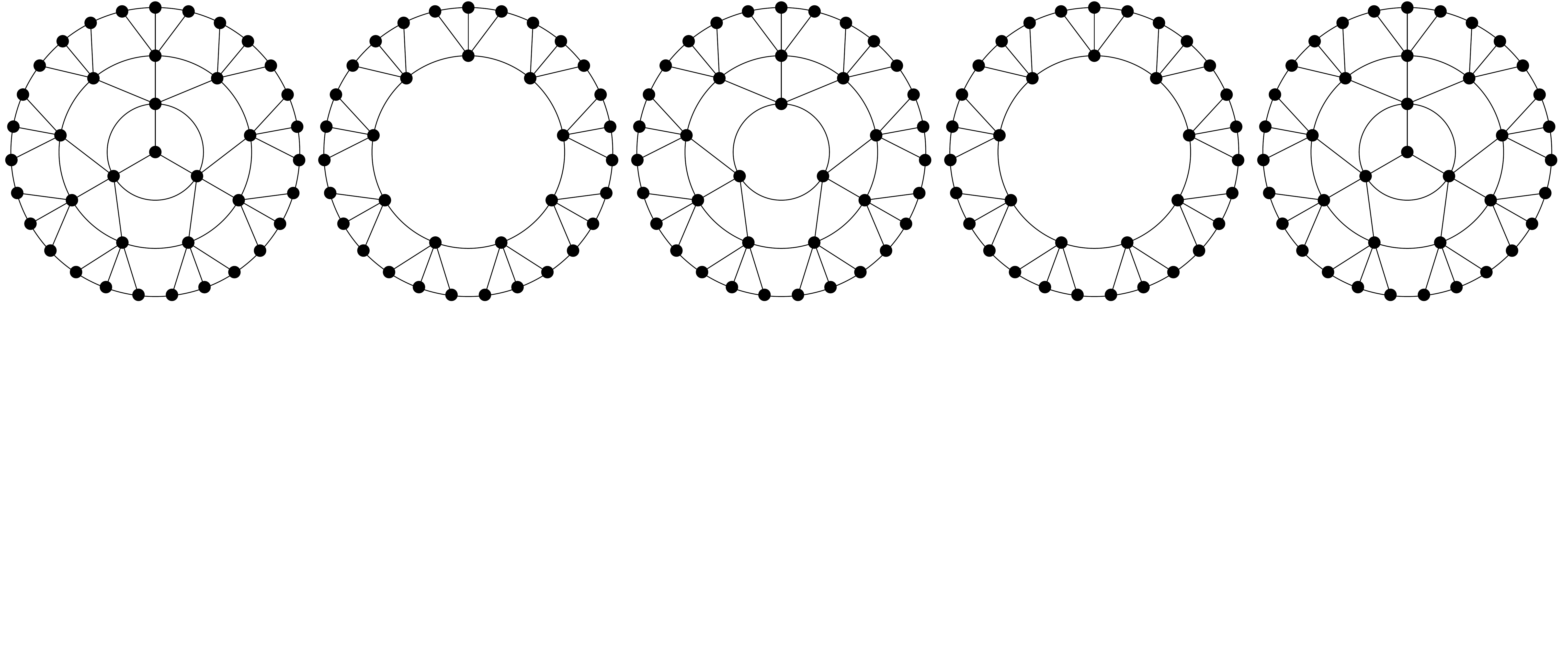}
\caption{A subgraph of $H_3$. Edges of the form $(t,w)(t',w')$ with $t=t'$ and $ww'\in E(W)$ are missing from the figure: these are all the edges joining corresponding vertices in consecutive components of the figure.}\label{Fig:onemeatball}
\end{center}
\end{figure}

Thus every vertex $t$ of $H_2$ gives rise to a double ray  in $H_3$, which consists of those vertices of $H_3$ that have $t$ as their first coordinate. Similarly, every vertex $w$ of $W$ gives rise to a cycle in $H_3$, the length of which depends on $h(w)$. We call the vertices on any such cycle \defi{cocircular}. Every ray of $T_2$ starting at $r$ gives rise to a copy of $W$, and if two such paths share their first $k$ vertices, then the corresponding copies of $W$ share their first $k$ levels of $h$. It is not hard to prove that $H_3$ is a hyperbolic graph, but we will omit the proof as we will not use this fact.




\medskip
We next construct $\Psi$ by glueing a sequence of finite subgraphs $M_n$  of $H_3$ along a ray   $S$. We could choose the subgraph $M_n$ to be a ball in $H_3$, but we found it more convenient to work with somewhat different subgraphs of $H_3$: we let $M_n$ be the finite subgraph of $H_3$ spanned by those vertices $(t,w)$ such that $w$ lies in a certain \defi{box} $B_n \subseteq W$ of $W$ defined as follows.
Consider a subpath $P_n$ of the bottom double-ray of $W$ of length $3\cdot2^n$, and let $B_n$ consist of those vertices $w$ that lie in or above $P_n$ (as drawn in \fig{Fig:ww}) and satisfy $h(w) \leq n$. 

This completes the definition of $M_n$. We let $S_n$ denote the vertices of $M_n$ corresponding to $P_n$, and we index the vertices of $S_n$ as $\{r(x),\,0\leq x\leq 3\cdot2^n\}$. Note that $S_n$ is a geodesic of $M_n$. We subdivide $S_n$ into three parts: $L_n:=\{r(x),0\leq x<2^n\}, m_n:=r(2^n)$ and $R_n:=\{r(x),2^n<x\leq3\cdot2^n\}$.
We define the \defi{ceiling} $F_n$ of $M_n$ to be its vertices of maximum height, i.e.\ the vertices $(t,w)\in V(M_n)$ with $h(w)=n$.


Finally, it remains to describe how to glue the $M_n$ together to form $\Psi$.
We start with a ray $S$, the first vertex of which we denote by $o$ and call the \defi{root} of $\Psi$. We glue $M_1$ on $S$ by identifying $S_1$ with the initial subpath of $S$ of length $|S_1|$. Then, for $n=2,3,\ldots$, we glue $M_n$ on $S$ in such a way that $L_n$ is identified with $R_{n-1}$ (where we used the fact that $|L_n|=|R_{n-1}|= 2^n$ by construction), $m_n$ is identified with the following vertex of $S$, and $R_n$ is identified with the subpath of $S$ following that vertex and having length $|R_n|=2^{n+1}$. 
Of course, we perform this identification in such a way that the linear orderings of $L_n$ and $R_n$ are given by the induced linear ordering of $S$. 
We let $\Psi$ denote the resulting graph. We think of $M_n$ as a subgraph of $\Psi$.

\subsection{Properties of $\Psi$}

By construction, for $j>i$ we have $M_i \cap M_j=\emptyset$ unless $j=i+1$, in which case $M_i \cap M_j= R_i = L_j \subset S$. The following fact is easy to see.
\labtequ{sepinf}{For every $n$, $R_n$ separates $L_n$ (and $o$) from infinity.}


The following property will be important for the proof of Liouvilleness.

\labtequ{roof}{There is a uniform lower bound $p>0$ for the probability $\PrII{v}{\tau_{F_n}<\tau_{S_n}}$ that \rw\ from any vertex of $L_n$ will visit the ceiling $F_n$ before returning to $S_n$.}

Indeed, we can let $p$ be the probability for \rw\ on $H_2$ starting at the root $o$ to never visit $o$ again; this is positive because $H_2$ is transient. Then \eqref{roof} holds because in a \rw\ from $S_n$ on $M_n$, any steps inside the copies of $H_2$ behave like \rw\ on $H_2$ until hitting $F_n$, and the steps `parallel' to $S_n$ do not have any influence.

\section{Hyperbolicity}

In this section we prove that $\Psi$ is hyperbolic in the sense of Gromov \cite{gromov}.

\begin{lem}
The graph $\Psi$ is hyperbolic, and has a one-point hyperbolic boundary.
\end{lem}
\begin{proof}

We claim that for every vertex $x\in V(\Psi)$, there is a unique 1-way infinite geodesic starting at $x$. Indeed, this geodesic $x_0 x_1 \ldots$, takes a step from $x_i$ towards the root of $T_3$ inside the copy of $H_2$ corresponding $x_i$ whenever such an edge exists in $\Psi$, and it takes a horizontal step in the direction of infinity whenever such an edge does not exist.


The hyperbolicity of $\Psi$ now follows from a well-known fact saying that a space is hyperbolic \iff\ any two geodesics with a common starting point are either at bounded distance or diverge exponentially in a certain sense; see \cite{short} for details. We skip the details here as in our case the condition is trivially satisfied do to the above claim.

\medskip
As all infinite geodesics eventually coincide with $S$, we also immediately have that the hyperbolic boundary of $G$ consists of just one point.
\end{proof}

\section{Transience} \label{secTrans}

In this section we prove that $\Psi$ is transient. We do so by displaying a flow from $o$ to infinity having finite Dirichlet energy; transience then follows from Lyons' criterion:

\begin{thm}[{T.~Lyons' criterion (see \cite{LyoSim} or \cite{LyonsBook})}] \label{lyoCr}
A graph $G$ is transient, \iff\ $G$ admits a flow of finite energy from a vertex to infinity.
\end{thm}

We refer the reader to \cite{LyonsBook} or \cite{UKtrans} for the basics of electrical networks needed to understand this theorem.

\begin{figure}[h!]
\begin{center}
\includegraphics[scale=0.6]{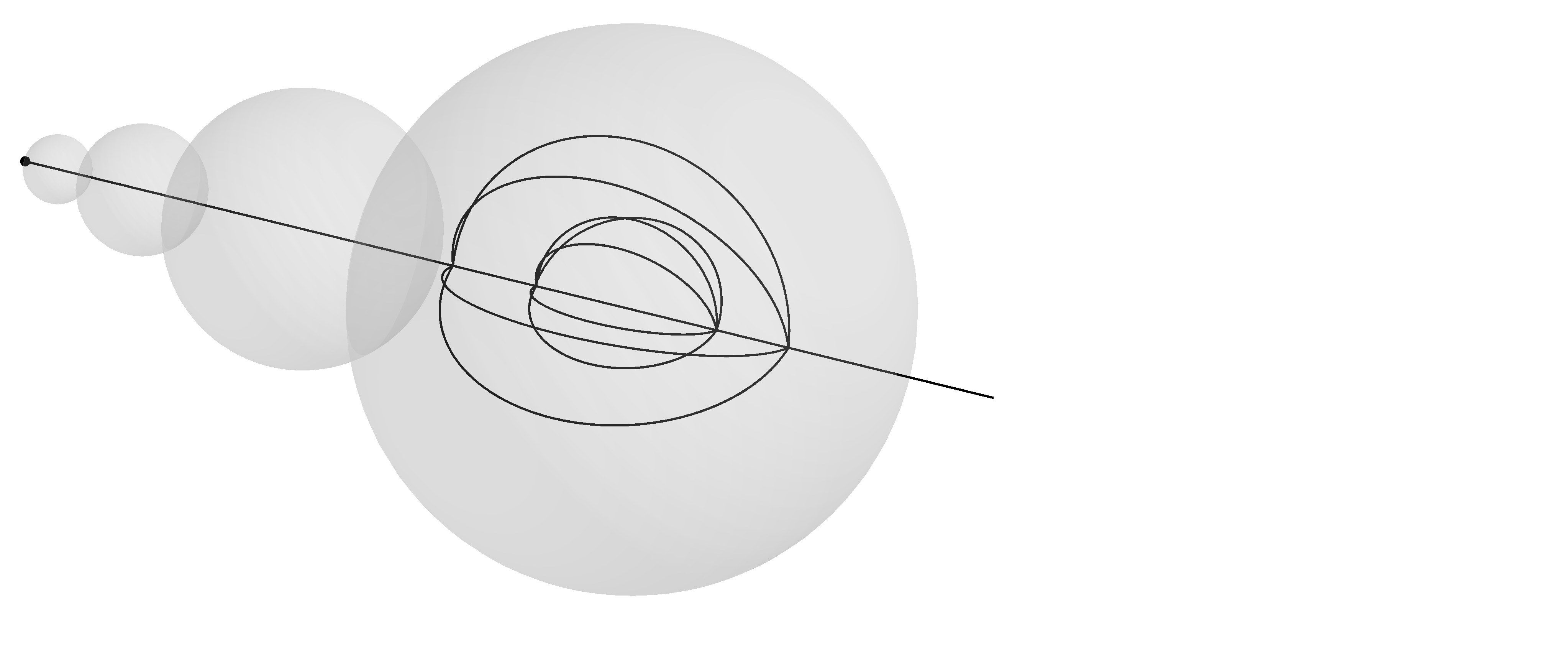}
\caption{The structure of the graph $\Psi$, with the `balls' intersecting along the ray and the flow inside the ball.}\label{Fig:meatballs}
\end{center}
\end{figure}

To construct this flow $f$, we start with the flow $t$ on the tree $T_3\subset H_2$ which sends the amount $3^{-n}$ through each 
directed edge of $T_3$ from a vertex of distance $n-1$ from the root to a vertex of distance $n$ from the root. Note that $t$ has finite Dirichlet energy.


Our flow $f$ will be as described in the introduction, that is, it is composed of flows $g(n)$ in $M_n$.
These flows flow from $L_n$ to $R_n$.
The flow $g(n)$ in turn is composed of \emph{`atomic'} flows, one for each $v\in L_n$.
Roughly, these atomic flows imitate $t$ from above for some levels, then use the edges parallel to $S_n$ to bring it `above' $R_n$, and then collect it back to (two vertices of) $S_n$ imitating $t$ in the inverse direction. A key idea here is that although the energy dissipated along the long paths parallel to $S_n$ is proportional to their length, by going up enough levels with the $t$-part of these flows, we can ensure that the flow $i$ carried by each such path is very small compared to its length $\ell$.
Thus its contribution $i^2 \ell$ to the Dirichlet energy can be controlled: although going up one level doubles $\ell$, and triples the number of long paths we have, each of them now carries $1/3$ of the flow, and so its contribution to the energy is multiplied by a factor of $1/9$. Thus all in all, we save a factor of $6/9$ by going up one 
more level -- and we have made the $M_i$ high enough that we can go up enough levels.

\medskip
We now describe $g(n)$ precisely. 
For every $n\in \N$, let us first enumerate the vertices of $L_n$ as $l^j= l_n^j$, with $j$ ranging from 1 to $|L_n|=2^n$, in the order they appear on $S_n$ as we move from the midpoint $m_n$ towards the root $o$. Likewise, we enumerate the vertices of $R_n$ as $r^j= r_n^j$, with $j$ ranging from 1 to $|R_n|= 2|L_n|$, in the order they appear on $S_n$ as we move from the midpoint $m_n$ towards infinity. Thus $r^1, l^1$ are the two neighbours of $m_n$ on $S$. We will let $g(n)$ be the union of $|L_n|$ subflows $g^j= g_n^j$, where $g^j$ flows from $l^j$ into $r^{2j}$ and $r^{2j-1}$. More precisely, $g^j$ sends $1/|L_n|=2^{-n}$ units of current out of $l^j$, and half as many units of current into each of $r^{2j}$ and $r^{2j-1}$.

We define $g^j$ as follows. In the copy of $H_2$ containing the source $l^j$ of $g^j$, we multiply the flow $t$ from above by the factor  $2^{-n}$, and truncate it after $j$  layers; we call this the out-part of $g^j$. Then, from each endpoint $x$ of that flow, we send the amount of flow that $x$ receives from $l^j$, which equals $2^{-n} 3^{-j}$, along the horizontal path $P_x$ joining $x$ to the copy $C_1$ of $H_2$ containing $r^{2j-1}$. We let half of that flow continue horizontally to reach the copy $C_2$  of $H_2$ containing $r^{2j}$; call this the middle-part of $g^j$. Finally, inside each of $C_1, C_2$, we put a copy of the out-part of $g^j$ multiplied by $1/2$ and with directions inverted; this is called the in-part of $g^j$. Note that the union of these three parts is a flow of intensity $2^{-n}$ from $l^j$ to $r^{2j}$ and $r^{2j-1}$, each of the latter receiving $2^{-n-1}$ units of current.

Let us calculate the energy $E(g^j)$. The contribution to  $E(g^j)$ by its out-part is bounded above by $2^{-2n}E(t)$ because that part is contained in the flow $2^{-n}t$. Similarly, the contribution of the in-part is half of the contribution of the out-part. The contribution of the middle-part is $3^{j} \cdot (2j+1) 2^{j} \cdot (2^{-n} 3^{-j})^2$: the factor $3^{j}$ counts the number of horizontal paths used by the flow, each of which has length $(2j+1)2^{j} $, and carries  $2^{-n} 3^{-j}$ units of current (except for its last $2^{j}$ edges, from $C_1$ to $C_2$, which carry half as much, but we can afford to be generous). Note that this expression equals $2^{-2n}(2j+1) (6/9)^{j}$, which is upper bounded by $k 2^{-2n}$ for some constant $k$.

Adding up these contributions, we see that $E(g^j)\leq K 2^{-2n}$ for some constant $K$ (which depends on neither $n$ nor $j$).

\medskip
Now let $g(n)$ be the union of the $2^n$ flows $g^j$. 
Note that $g^j, g^i$ are disjoint for $i\neq j$, and therefore the energy $E(g(n))$ of $g$ is just the sum $\sum_{j<2^n} E(g^j)$.
By the above bound, this yields $E(g(n)) \leq K 2^{-n}$.

Now let $f= \bigcup_{n\in \N} g(n)$ be the union of all the flows $g(n)$. Then $g(n), g(m)$ are  disjoint for $n\neq m$, because they are in different $M_i's$. Thus $E(f) = \sum_n E(g(n))\leq K$ is finite. 
Since $g(n)$ removes as much current from each vertex of $L_n$ as $g(n-1)$ inputs, $f$ is a flow from $o$ to infinity. 
Hence $\Psi$ is transient by Lyons' criterion (\Tr{lyoCr}).

\section{Liouville property} \label{secLiou}

In this section we prove that $\Psi$ is \defi{Liouville}, i.e.\ it admits no bounded non-constant harmonic functions.

We remark that a well-known theorem of Ancona \cite{AncNeg} states that in any non-amenable hyperbolic graph the hyperbolic boundary coincides with the Martin boundary. We cannot apply this fact to our case in order to deduce the Liouville property from the fact that our hyperbolic boundary is trivial, because our graph turns out to be amenable.

\medskip
We will use some elementary facts about harmonic functions that can be found e.g.\ in \cite{planarPB}. 

Let $h$ be a bounded non-constant harmonic functions on a graph $G$. We may assume that the range of $h$ is contained in $[0,1]$. Recall that, by the bounded martingale convergence theorem, if $(X_n)_{n \in \N}$ is a simple \rw\ on $G$, then $h(X_n)$ converges almost surely.
We call such a function $h$ \defi{sharp}, if this limit $\lim_n h(X_n)$ is either 0 or 1 almost surely. It is well-known that if a graph admits a bounded non-constant harmonic function, then it admits a sharp  harmonic function, see \cite[Section~4]{planarPB}.

So let us assume from now on that $h: V(\Psi) \to [0,1]$ is a sharp bounded harmonic function on $\Psi$.

We first recall some basic facts from \cite[Section~7]{planarPB}; we repeat some of the proofs for the convenience of the reader.

\begin{lem}\label{hprob}	
	If $h$ is a \shf, then $h(z)=\PrII{z}{\lim h(Z_n) =1}$ \fe\ vertex $z$, where $Z_n$ denotes a \rw\ from $z$.
\end{lem}

\begin{figure}[h!]
\begin{center}
\includegraphics[scale=0.6]{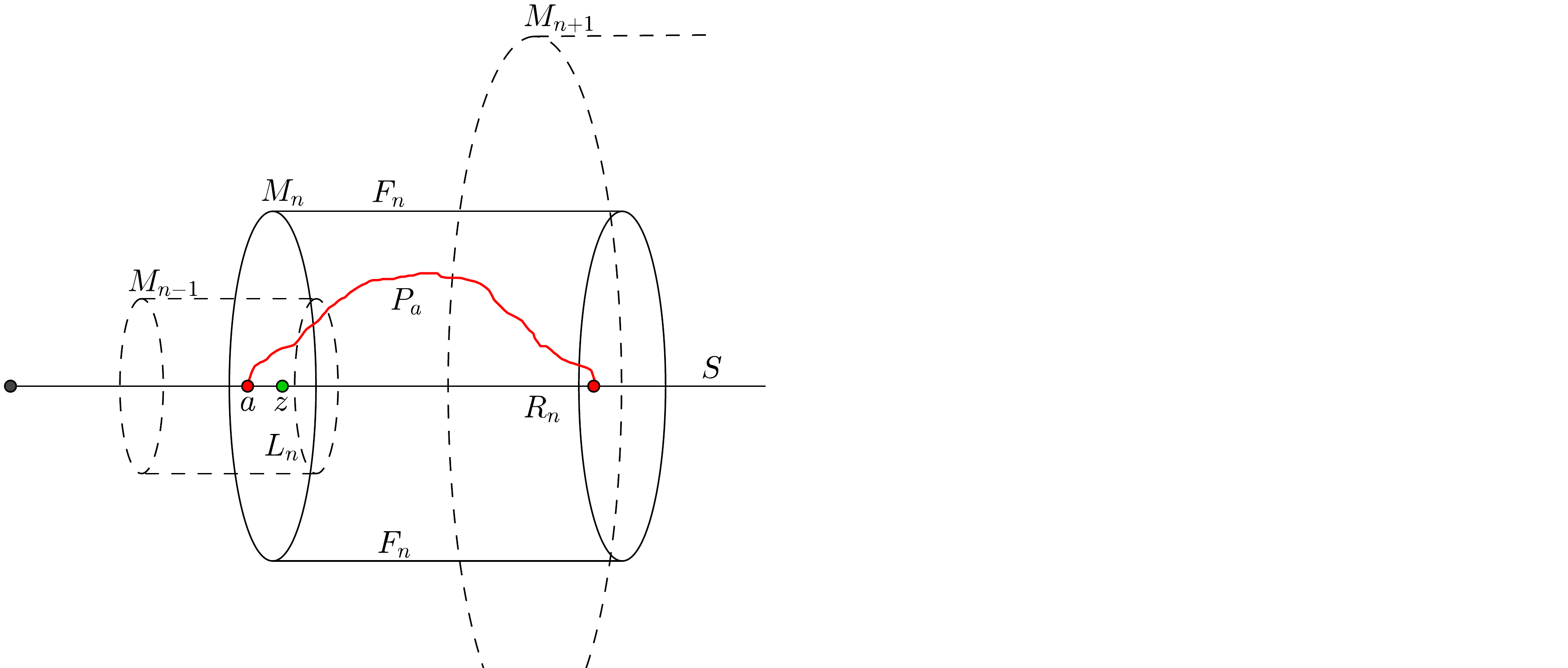}
\caption{The path $P_\alpha$ in the proof of the Liouville property.}\label{Fig:souvlaki}
\end{center}
\end{figure}

\begin{lem}\label{close1}	
	If $h$ is a \shf\ that is not constant, then \fe\ $\eps>0$ there are $a,z\in V$ with $h(a)<\eps$ and $h(z)>1-\eps$.
\end{lem}

Let $\ca$ be a shift-invariant event of our \rw, i.e.\ an event not depending on the first $n$ steps \fe\ $n$.
(The only kind of event we will later consider is the event $1^s$ that $s(Z_n)$ converges to 1, where $s$ is our fixed sharp harmonic function.)

For $r\in (0,1/2]$, let 
\labtequstar{
	$A_r:= \{v\in V \mid \PrI{\ca} > 1-r\}$ and\\ 
	$Z_r:= \{v\in V \mid \PrI{\ca} < r\}$.
}
Note that $A_r \cap Z_r=\emptyset$ \fe\ such $r$.

By \Lr{hprob}, if we let $\ca:= 1^s $ then we have $A_r= \{v\in V \mid s(v)> 1-r\}$ and $Z_r= \{v\in V \mid s(v)< r\}$.

\begin{lem}\label{Ade}
	\Fe\ $\eps,\del\in (0,1/2]$, and every $v\in A_\eps$, we have\\  $\PrII{v}{\text{visit $V \sm A_\del$}} < \eps/\del$. Similarly, \fe\ $v\in Z_\eps$, we have\\ $\PrII{v}{\text{visit $V \sm Z_\del$}} < \eps/\del$.
\end{lem}
\begin{proof}
	Start a \rw\ $(Z_n)$ at $v$, and consider a stopping time $\tau$ at the first visit to $V \sm A_\del$. If $\tau$ is finite, let $z=Z_\tau$ be the first vertex of random walk outside $A_\del$.  Since $z\not\in A_\del$, 
	the probability that $s(X_n)$ converges to 1 for a random walk $(X_n)$ starting from $z$ is at least $\delta$
by the definition of $A_\del$. Thus, subject to visiting $V \sm A_\del$, the event $ \ca$ fails with probability at least $\del$ since it is a shift-invariant event. But $ \ca$ fails with probability less than $\eps$ because $v\in A_\eps$, and so $\PrII{v}{\text{visit $V \sm A_\del$}} < \eps/\del$ as claimed. 
	
	The second assertion follows by the same arguments applied to the complement of $\ca$.
\end{proof}

\begin{cor}\label{path}
	If a \rw\ from $v\in A_\eps$ (respectively, $v\in Z_\eps$) visits a set $W\subset V$ with probability at least $\kappa$, then \ti\ a ${v}$--${W}$~path all vertices of which lie in $A_{\eps/\kappa}$ (resp.\ $Z_{\eps/\kappa}$).
\end{cor}
\begin{proof}
	Apply \Lr{Ade} with $\del= \eps/\kappa$. Then the probability that random walk always stays within $A_{\eps/\kappa}$ is larger than $1-\kappa$. Hence there is a nonzero probability that random walk meets $W$ and along its trace only has vertices from $A_{\eps/\kappa}$. 
\end{proof}

\medskip

Easily, $h$ is uniquely determined by its values on the skewer $S$. Indeed, for every other vertex $v$, note that \rw\ $X_n$ from $v$ visits $S$ almost surely, and so $h(v)= \Ex h(X_{\tau(R)})$, where $\tau(S)$ denotes the first hitting time of $S$ by $X_n$. The same argument implies that 

\labtequ{radial}{$h$ is radially symmetric, i.e.\ for every two \coc\ vertices $v,w$, we have $h(v)=h(w)$.}

Indeed, this follows from the fact that \coc\ vertices have the same hitting distribution to $S$, which is easy to see (for any vertex on a circle, \rw\ has the same probability to move to some other circle).


We claim that, given an arbitrarily small $\eps>0$, all but finitely many of the $L_n$  contain a vertex in $Z_\eps$. 

Indeed if not, then since \rw\ from $o$ has to visit all $L_n$ by transience and \eqref{sepinf} (where we use the fact that $L_n=R_{n-1}$), we would have $\PrI{\lim h(X_n)=1} =0$ for \rw\ $X_n$ from $o$.
But that probability is equal to $h(o)$ by \Lr{hprob}, and if it is zero, then using \Lr{hprob} again easily implies that $h$ is identically zero, contrary to our assumption that it is not constant.

Similarly, all but finitely many of the $L_n$  contain a vertex in $A_\eps$.
Thus we can find a late enough $M_n$ \st\ $L_n$ contains a vertex $a\in A_\eps$ as well as a vertex $z\in Z_\eps$. 
We assume that $a$ and $z$ are the last vertices of $L_n$ (in the ordering of $L_n$ induced by the well-ordering of $S$) that are in $A_\eps$ and $Z_\eps$, respectively. 
Assume without loss of generality that $a$ appears before $z$ in the ordering of $L_n$.

Note that, since $R_n$ separates $a$ from infinity \eqref{sepinf}, \rw\ from $a$ visits $R_n$ almost surely. Thus we can apply \Cr{path} with $W:= R_n$ and $\kappa=1$ to obtain an $a$--$R_n$~path $P_a$ with all its vertices in $A_\eps$. 
We may assume that $P_a \subset M_n$  by taking a subpath contained in $M_n$ if needed. Indeed, $P_a$ can meet $L_n$ only in vertices that are not past $a$ in the linear ordering of $L_n$.

Let $O_a$ denote the set of vertices $\{x=(t,w)\in M_n \mid \text{ there is } (t',w')\in V(P_a) \text{ with } w'=w \}$  obtained by `rotating' $P_a$ around $S$. By \eqref{radial}, we have $O_a\subset A_\eps$ since $P_a\subset A_\eps$. Note that $O_a$ separates $z$ from the ceiling $F_n$ of $M_n$. But as \rw\ from $z\in Z_\eps$ visits $F_n$ before returning to $S$ with probability uniformly bounded below by \eqref{roof}, we obtain a contradiction to  \Lr{Ade} with $\del=1/2$ for $\eps$ small enough compared to that bound.

\section{A transient hyperbolic graph with no transient subtree} \label{subtree}

In this section we explain how our souvlaki construction can be slightly modified so that it does not contain any transient subtrees but remains transient and hyperbolic (and Liouville). This answers a question of I.~Benjamini (private communication). The question is motivated by the fact that it is not too easy to come up with transient graphs that do not have transient subtrees \cite{BeSchrChee}. 

We start with a very fast growing function $f:\mathbb{N}\to \mathbb{N}$, whose precise definition we reveal at the end of the proof. 
Roughly speaking, we will attach a sequence of finite graphs $(M_{f(n)})_{ n\in \mathbb{N}}$ similar to the `meatballs' from above to a ray $S$ (the `skewer') in such a way that most of the intersection of $S$ with a fixed meatball is not contained in any other meatball.
Formally, we let $P_m$ be the `bottom path' of $M_m$ as defined in \Sr{secConstr}, and we  tripartition $P_{f(n)}$  as follows:
Let $L_n$ consist of the first $2^n$ vertices on $P_{f(n)}$, and $R_n$ consist of the last $2^{n+1}$.
The set of the remaining vertices of $P_{f(n)}$ we denote by $Z_n$, which by our choice of $f$ will be much larger than $R_n$. 
As before, we glue the $M_{f(n)}$ on $S$ by identifying $P_n$ with a subpath of $S$. 
We start by glueing $M_{f(1)}$ on the initial segment of $S$ of the appropriate length. 
Then we recursively glue the other $M_{f(n)}$ in such a way that $L_n$ is identified with $R_{n-1}$.
We call the resulting graph $\bar \Psi$.

\begin{thm}\label{trans_tree}
	$\bar \Psi$ is a bounded degree transient gromov-hyperbolic graph that does not contain a transient subtree.
\end{thm}


\begin{proof}
	The hyperbolicity and the transience $\bar \Psi$ can be proved by the arguments we used for the original souvlaki $\Psi$.
	So it remains to show that $\bar \Psi$ does not have a transient subtree. 
	
	Let $T$ be any subtree of $\bar \Psi$. 
	We want to prove that $T$ is not transient.  Easily, we may assume that $T$ does not have any degree 1 vertices. We will show that the
	following quotient $Q$ of $T$ is not transient:
	for each $n$, we identify all vertices in $L_n$ to a new vertex $v_n$.
	
	Note that the vertices $v_n$ and $v_{n+1}$ are cut-vertices of  $Q$; let $Q_n$ be the
	block of $Q$ incident with both $v_n$ and $v_{n+1}$ containing these two vertices (a \defi{block} is a maximally 2-connected subgraph).
	We will show that in $Q_n$ the effective resistence from $v_n$ to
	$v_{n+1}$ is bounded away from 0, from which the recurrence of $T$ will follow using Lyons' criterion.

	Let $d=|L_{n+1}|$. We claim that there is some constant $c=c(d)$ only depending
	on $d$ such that there are at most $c$ vertices of $Q_n$ with a degree
	greater than 2: indeed, $Q_n \sm \{v_n, v_{n+1}\}$ is a forest with $d(v_n) + d(v_{n+1})$ leaves.
	
	Next, we observe that $Q_n$ has maximum degree at most $d$.
	Furthermore, the distance between $v_n$ and $v_{n+1}$ in $Q_n$ is at least $Z_n$, which ---by the choice of $f$--- is huge compared to $d$ and so also compared to $c$.
	Hence it remains to prove the following:
	
	\begin{lem}
		For every constant $C$ and every $m$ there is some $s= s(m,C)$, such
		that for every finite graph $K$ with maximum degree at most $C$ and at most $C$ vertices of degree greater than $2$, and any two vertices $x,y$ of $K$ at	distance $x$ and $y$ at least $s$, the effective resistence
		between $x$ and $y$ in $K$ is at least $m$.
	\end{lem}
	
	\begin{proof}
		We start with a large $R\in \N$ the value of which we reveal later, and set $s=R \cdot C$.
		
		Let $K'$ be the graph obtained from $K$ by contracting all
		vertices of degree $2$. We colour an edge of $K'$ black if it is
		subdivided at least $R$ times in $K$.
		Note that $K'$ has at most $C$ vertices. 
		Thus every $x$-$y$-path in $K'$ has
		length at most $C$, but in $K$ any such path has length at least $s$. Therefore each $x$-$y$-path in $K'$ contains a black edge. Hence in $K'$ there is an $x$-$y$-cut
		consisting of black edges only. This cut has at most $C^2$ edges.
		Thus the effective resistence in $K$ between $x$ and $y$ is at least the
		one of that cut considered as a set of paths in $K$, which is as large as we want:
		indeed, we can pick $R$ so large that the latter resistance exceeds $m$.
	\end{proof}
	
	Now we reveal how large we have picked $f(n)$:
	recall that $d=2^{n+1}$ and that $|Z_n|=f(n)-3\cdot 2^n$. 
	We pick $f(n)$ large enough  that $|Z_n|\geq s(1,max(c(d),d))$, where $s$ is as given by the last lemma. 
	With these choices the effective resistence between $v_n$ and $v_{n+1}$ in $Q_n$ is at least $1$. So $Q$ cannot be transient by Lyons' criterion (\Tr{lyoCr}). By Rayleigh's monotonicity law \cite{LyonsBook}, $T$ is recurrent too. 
	
\end{proof}

\subsection{Another transient hyperbolic graph with no transient subtree} \label{subtree2}

We now sketch another construction of a transient hyperbolic graph with no transient subtree, provided by Tom Hutchcroft and Asaf Nachmias (private communication).
\medskip

Let $[0,1]^3$ be the unit cube. For each $n\geq0$, let $D_n$ be the set of closed dyadic subcubes of length $2^{-n}$.
For each $n\geq 0$, let $G_n$ be the graph with vertex set $\bigcup_{i=0}^n D_i$, and where two cubes $x$ and $y$ are adjacent if and only if
\begin{itemize}
	\item
	$x \supset y$, $x\in D_i$ and $y\in D_{i+1}$ for some $i\in \{0,\ldots n-1\}$, 
	\item $y \supset x$, $y\in D_i$ and $x\in D_{i+1}$ for some $i\in \{0,\ldots n-1\}$, or
	\item $x,y\in D_i$ for some $i\in \{0,\ldots, n\}$ and $x\cap y$ is a square. 
\end{itemize}
Then the graphs $G_n$ are uniformly Gromov hyperbolic and, since the subgraph of $G_n$ induced by $D_n$ is a cube in $\mathbb{Z}^3$ (of size $4^n$), the effective resistance between two corners this cube are bounded above uniformly in $n$. Moreover, the distance between these two points in $G_n$ is at least $n$.

Let $T$ be a binary tree, and let $G$ be the graph formed by replacing each edge of $T$ at height $k$ from the root with a copy of $G_{3^k}$, so that the endpoints of each edge of $T$ are identified with opposite corners in the corresponding copy of $D_{3^k}$. Since the graphs $G_n$ are uniformly hyperbolic and $T$ is a tree, it is easily verified that $G$ is also hyperbolic. The effective resistance from the root to infinity in $G$ is at most a constant multiple of the effective resistance to infinity of the root in $T$, so that $G$ is transient. 
However, $G$ does not contain a transient tree, since every tree contained in $G$ is isomorphic to a binary tree in which each edge at height $k$ from the root has been stretched by at least $3^k$, plus some finite bushes.

\section{A transient graph with no embeddable transient subgraph} \label{Kr}

We say that a graph $H$ has a graph $K$ as a \defi{minor}, if $K$ can be obtained from $H$ by deleting vertices and edges and by contracting edges. Let $K^r$ denote the complete graph on $r$ vertices.

\begin{prop} \label{propKr}
There is a transient bounded degree graph $G$ such that every transient subgraph of $G$ has  a $K^r$ minor for every $r\in \N$.
\end{prop}

In particular, $G$ has no transient subgraph that embeds in any surface of finite genus.

We now construct this graph $G$.
We will start with the infinite binary tree with root $o$, and  replace each edge at distance $r$ from $o$ with a gadget $D_{2^r}$ which we now  define. 
Given $n\ (=2^r)$, the vertices of $D_n$ are organized in $2n+1$ levels numbered $-n, \ldots,-1,0,1, \ldots, n$. Each level $i$ has $2^{n-|i|}$ vertices,  and two levels $i,j$ form a complete bipartite graph whenever $|i-j| = 1$; otherwise there is no edge between levels $i,j$. 
Any edge of $D_n$ from level $i\geq 0$ to level $i+1$ or from level $-i$ to level $-(i+1)$ is given a resistance equal to $2^{n-|i|}$ (we will later subdivide such edges into paths
of that many edges each having resistance 1). With this choice, the effective resistance $R_i$ between levels $i$ and $i+1$ of $D_n$ is $2^{n-|i|}$ divided by the number of edges between those two levels, that is, $R_i=\frac{2^{n-|i|}}{2^{n-|i|} 2^{n-|i|-1}} =2^{-n+|i|+1}$, and so the effective resistance in $D_n$ between its two vertices at levels $n$ and $-n$ is $O(1)$

Let $G'$ be the graph obtained from the infinite binary tree with root $o$ by replacing each edge $e$ at distance $n$ from $o$ with a disjoint copy of $D_n$, attaching the two vertices at levels $n$ and $-n$ of $D_n$ to the two end-vertices of $e$. We will later modify $G'$ to obtain a bounded degree $G$ with similar properties satisfying Proposition~\ref{propKr}.

Note that as $D_n$ has effective resistance $O(1)$, the graph $G'$ is transient by Lyons' criterion. 

We are claiming that if $H$ is a transient subgraph of $G'$, then $H$ has a $K^r$ minor for every $r\in \N$. 

This will follow from the following basic fact of finite extremal graph theory \cite{Mader,Kostochka,diestelBook05}

\begin{thm}\label{mader} 
	For every $r\in \N$ there is a constant $c_r$ \st\  every graph of average degree at least $c_r$ has a $K^r$ minor.
\end{thm}

\begin{lem} \label{lemKr}
If $H$ is a transient subgraph of $G'$, then $H$ has a $K^r$ minor for every $r\in \N$.	
\end{lem}
\begin{proof}
Suppose that $H$ has no $K^r$ minor for some $r$, and fix any $m\in \N$.
 Consider for every copy $C$ of the gadget $D_n$ in $G'$
 where $n >m$, the bipartite subgraph $G_m=G_m(C)$ of $H$ spanned by levels
 $m$ and $m+1$ of $C \cap H$. By Theorem~\ref{mader}, the average degree of
 $G_m$ is at most $c_r$. Thus, if we identify each of the partition
 classes of $G_m$ into one vertex, we obtain a graph with 2 vertices
 and at most $\frac{3}{2}  2^{n-m} c_r$ parallel edges, each of resistance $2^{n-m}$,
 so that the effective resistance of the contracted graph is greater than $1/c_r =:C_r$. 

  Now repeating this argument for $m+1, m+2, \ldots$, 
 we see that the effective resistance between the two partition classes of $G_{m+k}$ (which is edge-disjoint to $G_m$) is also at least the same constant $C_r$.
 This easily implies that the effective resistance between the two endvertices of $C\cap H$ for any copy $C$ of $D_n$ is $\Omega(n)$. Since $G'$ has $2^r$ copies of $D_{2^r}$ at each `level' $r$, we obtain that the effective resistance from $o$ (which we may assume without loss of generality to be contained in $H$) to infinity in $H$ is $\Omega(\sum_r 2^r/2^r) = \infty$. 
 
 Thus $H$ can have no electrical flow from a vertex to infinity, and
 by Lyons' criterion (\Tr{lyoCr}) it is not transient.
\end{proof}

Recall that the edges of $G'$ had resistances greater than 1. By replacing each edge of resistance $k$ by a path of length $k$ with edges having resistance 1, we do not affect the transience of $G'$. We now modify $G'$ further into a graph $G$ of bounded degree, which will retain the desired property.

Let $x$ be a vertex of some copy $C$ of $D_n$, at some level $j\neq n,-n$ of $C$. Then $x$ sends edges to the two neighbouring levels $j\pm 1$. Each of those levels $L, L'$, sends $2^{k\pm 1}$ edges to $x$ for some $k$. Now disconnect all the edges from  $L$ to $x$, attach a binary tree $T_L$ of depth $k \pm 1$ to $x$, and then reconnect those edges, one at each leaf of $T_L$. Do the same for the other level $L'$, attaching a new tree $T_{L'}$ of appropriate depth to $x$. Note that after doing this for every such $x$, the graph $G$ obtained has maximum degree 6 (we do not need to modify the vertices at levels $n,-n$ in $C$, as they already had degree 6.

Now let's check that $G$ is still transient, by considering the obvious flow to infinity. Each new tree of the form $T_L$ we attached has effective resistance from its root to the union of its leaves less that 1. Moreover, between any two levels of size about $2^k$ in some $D_n$ we introduced as many such trees as there are vertices in the levels. An easy calculation yields that the extra effective resistance we introduced between two levels is about $2/2^k$; hence the total resistance we introduced to each $D_n$ taking into account all its levels is bounded by a constant. Thus the effective resistance of each $D_n$ remains bounded by a constant (independent of $n$), and so $G$ is still transient.

Note that $G'$ can be obtained from $G$ by contracting edges. Thus any transient $H \subseteq G$ has a transient minor $  H' \subseteq G'$, because contracting edges preserves transience by Lyons' criterion. As we have proved that $H'$ has a $K^r$ minor (\Lr{lemKr}), so does $H$ as any minor of $H'$ is a minor of $H$.

\medskip
Despite Proposition~\ref{propKr}, the following remains open
\begin{que}[I.~Benjamini (private communication)] Does every bounded-degree transient graph have a transient subgraph which is sphere-packable in $R^3$?
\end{que}

\section{Problems}

It is not hard to see that our hyperbolic souvlaki $\Psi$ is \defi{amenable}, that is, we have $\inf_{\emptyset\neq S\subset \Psi\text{ finite}}\dfrac{|\partial S|}{|S|}=0$,
where $\partial S=\{v\in V(\Psi)\setminus S\mid\text{there exists }w\in S\text{ adjacent to }v\}$. We do not know if this is an essential feature:

\begin{problem}
Is there a non-amenable counterexample to \Cnr{bsconj}?  
\end{problem}

Similarly, one can ask
\begin{problem}
	Is there a 	non-amenable, hyperbolic graph with bounded-degrees, $C$-dense infinite geodesics, and the Liouville property, the hyperbolic boundary of which consists of a single point? 
\end{problem}
Here we did not ask for transience as it is implied by non-amenability \cite{BeSchrChee}.

We conclude with further questions asked by I.~Benjamini (private communication)
\begin{problem}
Is there a uniformly transient counterexample to \Cnr{bsconj}?  
Is there an 1-ended counterexample?
\end{problem}
Here \defi{uniformly transient} means that there is an upper bound on the effective resistance between any vertex of the graph and infinity.

\bibliographystyle{plain}
\bibliography{collective}

\end{document}